\documentclass[11 pt]{article}
\usepackage{amsmath,amssymb,amsfonts,amsthm,graphicx,cite,epsfig,makecell, array,mathrsfs}
\usepackage[varg]{pxfonts}
\usepackage{wrapfig}
\usepackage{enumitem}
\usepackage{epstopdf,footnote,multirow,booktabs,bigints,pstool,booktabs,caption,subcaption ,lipsum}
\usepackage[table]{xcolor}
\usepackage[colorlinks = true,
            linkcolor = red,
            citecolor = blue]{hyperref}
\usepackage[symbol]{footmisc}
\usepackage{titlesec, float}
\usepackage[export]{adjustbox}
\theoremstyle{plain}
\newtheorem{thm}{Theorem}[section]
\newtheorem{lm}[thm]{Lemma}

\theoremstyle{definition}
\newtheorem{de}[thm]{Definition}
\newtheorem{co}[thm]{Corollary}

\theoremstyle{remark}
\newtheorem{re}{\sc \textbf{Remark}}

\numberwithin{equation}{section}
\newcolumntype{?}{!{\vrule width 1.4pt}}

\def\correspondingauthor{\footnote{Corresponding author.}}
\titlelabel{\thetitle.\quad}
\renewenvironment{abstract}
               {\list{}{\rightmargin\leftmargin}%
                \item[\textbf{\hspace{8.6mm}Abstract ---}]\relax}
               {\endlist}
\DeclareUrlCommand{\url}{%
    \def\UrlLeft##1\UrlRight{\underline{##1}}}

\date{}
\title{New Subclasses of Analytic and Bi-Univalent Functions Endowed with Coefficient Estimate Problems}
\author{Feras Yousef$^{1,}$\correspondingauthor{}\,\,,\, Somaia Alroud$^{2}$,  Mohamed Illafe$^{3}$ \vspace{0.1in}\\ 
\footnotesize{$^{1,2}$Department of Mathematics, The University of Jordan, Amman 11942, Jordan. }  \\
 \footnotesize{$^{3}$School of Engineering, Math, $\&$ Technology, Navajo Technical University, Crownpoint, NM 87313, USA. }  \\
 \footnotesize{e-mail: $^1$fyousef@ju.edu.jo}, \,$^2$somaiaroud41@gmail.com,\, $^3$millafe@navajotech.edu}
\begin{document}
\maketitle
\begin{abstract} Inspired by the recent works of Srivastava et al. \cite{C27}, Frasin and Aouf \cite{C28}, and \c{C}a\u{g}lar et al. \cite{C29}, we introduce and investigate in the present paper two new general subclasses of the class consisting of normalized analytic and bi-univalent functions in the open unit disk  $\mathbb{U}=\{z\in\mathbb{C}:\left\vert z\right\vert <1\}$. For functions belonging to these general subclasses introduced here, we obtain estimates on the Taylor-Maclaurin coefficients $|a_{2}|$ and $|a_{3}|$. Several connections to some of the earlier known results are also pointed out. The results presented in this paper would generalize and improve those in related works of several earlier authors. \\
\end{abstract}



\section{Introduction and definitions}
Let $\mathcal{A}$ denote the class of all analytic functions $f$ defined in the open unit disk $\mathbb{U}=\{z\in\mathbb{C}:\left\vert z\right\vert <1\}$ and normalized by the conditions $f(0)=0$ and $f^{\prime}(0)=1$. Thus
each $f\in\mathcal{A}$ has a Taylor-Maclaurin series expansion of the form:
\begin{equation} \label{ieq1}
f(z)=z+\sum\limits_{n=2}^{\infty}a_{n}z^{n}, \ \  (z \in\mathbb{U}).
\end{equation}

Further, let $\mathcal{S}$ denote the class of all functions $f \in\mathcal{A}$ which are univalent in $\mathbb{U}$ (for details, see \cite{Duren}; see also some of the recent investigations \cite{C1,C2,C4,C5}).

Two of the important and well-investigated subclasses of the analytic and univalent function class $\mathcal{S}$ are the class
$\mathcal{S}^{\ast}(\alpha)$ of starlike functions of order $\alpha$ in $\mathbb{U}$ and the class $\mathcal{K}(\alpha)$ of convex functions of order $\alpha$ in $\mathbb{U}$. By definition, we have

\begin{equation}
\mathcal{S}^{\ast}(\alpha):=\left\{ f: \ f \in \mathcal{S} \ \ \text{and} \ \ \mbox{Re}\left\{ \frac{zf^{\prime }(z)}{f(z)}\right\} >\alpha,\quad (z\in \mathbb{%
U}; 0\leq \alpha <1) \right\},  \label{d1}
\end{equation}%
and
\begin{equation}
\mathcal{K}(\alpha):=\left\{ f: \ f \in \mathcal{S} \ \ \text{and} \ \ \mbox{Re}\left\{ 1+\frac{zf^{\prime \prime }(z)}{f^{\prime }(z)}\right\} >\alpha,\quad (z\in \mathbb{%
U}; 0\leq \alpha <1) \right\}. \label{d2}
\end{equation}%

It is clear from the definitions (\ref{d1}) and (\ref{d2}) that $\mathcal{K}(\alpha) \subset \mathcal{S}^{\ast}(\alpha)$. Also we have

\begin{equation}
f(z) \in \mathcal{K}(\alpha) \ \ \text{iff} \ \ zf^{\prime}(z) \in \mathcal{S}^{\ast}(\alpha),
\end{equation}%
and
\begin{equation}
f(z) \in \mathcal{S}^{\ast}(\alpha) \ \ \text{iff} \ \ \int_{0}^{z} \frac{f(t)}{t} dt =F(z) \in \mathcal{K}(\alpha).
\end{equation}%

It is well-known that, if $f(z)$ is an univalent analytic function from a domain $\mathbb{D}_{1}$ onto a domain $\mathbb{D}_{2}$, then the inverse function $g(z)$ defined by
\[
g\left(f(z)\right)=z, \ \ (z\in \mathbb{D}_{1}),
\]
is an analytic and univalent mapping from $\mathbb{D}_{2}$ to $\mathbb{D}_{1}$. Moreover, by the familiar Koebe one-quarter theorem (for details, see \cite{Duren}), we know that the image of $\mathbb{U}$ under every function $f\in \mathcal{S}$ contains a disk of radius $\frac{1}{4}$.

 According to this, every function $f\in \mathcal{S}$ has an inverse map $f^{-1}$ that satisfies the following conditions:
\begin{center}
$f^{-1}(f(z))=z \ \ \ (z\in \mathbb{U}),$
\end{center}
and
\begin{center}
$f\left(f^{-1}(w)\right)=w$ $\ \ \ \left( |w|<r_{0}(f);r_{0}(f)\geq\frac{1}{4}\right)$.
\end{center}

In fact, the inverse function is given by
\begin{equation} \label{ieq2}
f^{-1}(w)=w-a_{2}w^{2}+(2a_{2}^{2}-a_{3})w^{3}-(5a_{2}^{3}-5a_{2}a_{3}+a_{4})w^{4}+\cdots.
\end{equation}

A function $f\in \mathcal{A}$ is said to be bi-univalent in $\mathbb{U}$ if both $f(z)$ and $f^{-1}(z)$ are univalent in $\mathbb{U}$. Let $\Sigma$ denote the class of bi-univalent functions in $\mathbb{U}$ given by (\ref{ieq1}). Examples of functions in the class $\Sigma $ are
\[
\frac{z}{1-z}, \ -\log(1-z), \ \frac{1}{2}\log\left(\frac{1+z}{1-z}\right), \cdots.
\]

It is worth noting that the familiar Koebe function is not a member of $\Sigma $, since it maps the unit disk $\mathbb{U}$ univalently onto the entire complex plane except the part of the negative real axis from $-1/4$ to $-\infty$. Thus, clearly, the image of the domain does not contain the unit disk $\mathbb{U}$. For a brief history and some intriguing examples of functions and characterization of the class $\Sigma$, see Srivastava et al. \cite{C27}, Frasin and Aouf \cite{C28}, and Yousef et al. \cite{C3}.

In 1967, Lewin \cite{C21} investigated the bi-univalent function class $\Sigma $ and showed that $|a_{2}|<1.51$. Subsequently, Brannan and Clunie \cite{C22} conjectured that  $|a_{2}|\leq \sqrt{2}.$ On the other hand, Netanyahu \cite{C23} showed that $\underset{f\in \Sigma }{\max }$ $|a_{2}|=\frac{4}{3}.$ The best known estimate for functions in $\Sigma $ has been obtained in 1984 by Tan \cite{C24}, that is, $|a_{2}|<1.485$. The coefficient estimate problem for each of the following Taylor-Maclaurin coefficients $|a_{n}|$ $(n\in \mathbb{N}\backslash \{1,2\})$ for each $f\in \Sigma$  given by (\ref{ieq1}) is presumably still an open problem.

\newpage

Brannan and Taha \cite{C25} introduced certain subclasses of a bi-univalent function class $\Sigma$ similar to the familiar subclasses $\mathcal{S}^{\ast}(\alpha)$ and $\mathcal{K}(\alpha)$ of starlike and convex functions of order $\alpha$ ($0\leq \alpha <1$), respectively (see \cite{C26}). Thus, following the works of Brannan and Taha \cite{C25}, for $0\leq \alpha <1,$ a function $f\in \Sigma $ is in the class $\mathcal{S}_{\Sigma}^{\ast }\left( \alpha \right) $ of bi-starlike functions of order $\alpha$;
or $\mathcal{K}_{\Sigma }\left( \alpha \right)$ of bi-convex functions of order $\alpha$ if both $f$ and $f^{-1}$ are respectively starlike or convex
functions of order $\alpha.$ Also, a function $f\in \mathcal{A}$ is in the class $\mathcal{S}_{\Sigma }^{\ast }[\alpha]$ of strongly bi-starlike functions of order $\alpha \left( 0\leq \alpha <1 \right)$ if each of the following conditions is satisfied:
\[
f\in \Sigma \ \ \text{and} \ \ \left\vert \arg \left( \frac{zf^{\prime }(z)}{%
f(z)}\right) \right\vert <\frac{\alpha \pi }{2} \quad  \left( 0\leq \alpha <
1,z\in \mathbb{U}\right)
\]
and
\[
\left\vert \arg \left( \frac{wg^{\prime }(w)}{g(w)}\right)
\right\vert <\frac{\alpha \pi }{2} \quad  \left( 0\leq \alpha <1,w\in \mathbb{U}\right),
\]
where $g$ is the extension of $f^{-1}$ to $\mathbb{U}$.

Recently, many researchers have introduced and investigated several interesting subclasses of the bi-univalent function class $\Sigma$ and they have found non-sharp estimates on the first two Taylor-Maclaurin coefficients $|a_{2}|$ and $|a_{3}|$.  In fact, the aforecited work of
Srivastava et al. \cite{C27} essentially revived the investigation of various subclasses of the bi-univalent function class $\Sigma$ in recent years; it was followed by such works as those by Frasin and Aouf \cite{C28}, Xu et al. \cite{C210}, \c{C}a\u{g}lar et al. \cite{C29}, and others (see, for example, \cite{C211,C212} and \cite{C213}). Motivated by the aforementioned works, the main object of the present investigation is to introduce two new subclasses of the function class $\Sigma$ and find estimates on the coefficients $|a_{2}|$ and $|a_{3}|$ for functions in these new subclasses of the function class $\Sigma$ employing the techniques used earlier by Srivastava et al. \cite{C27}. We also extend and improve the aforementioned results of Srivastava et al. \cite{C27}, Frasin and Aouf \cite{C28}, and \c{C}a\u{g}lar et al. \cite{C29}. Various known or new special cases of our results are also pointed out.

The following lemma will be required in order to derive our main results.

\begin{lm} \label{lem} \cite{C214}
If $p\in \mathcal{P}$, then $|c_{k}|\leq 2$ for each k , where $\mathcal{P}$ is the family of
all functions $p$ analytic in $\mathbb{U}$ for which $\mbox{Re} \left(p(z)\right)>0, p(z)=1+c_{1}z+c_{2}z^{2}+\cdots$ for $z\in \mathbb{U}.$
\end{lm}


\section{Coefficient bounds for the function class $\mathscr{B}_{\Sigma }^{\mu }(\alpha,\lambda,\delta)$}
We begin this section by introducing the following subclass of the function class $\Sigma$.

\begin{de} \label{def21}
For $\lambda \geq 1,\mu \geq 0, \delta \geq 0$ and $0 < \alpha \leq 1$, a function $f\in \Sigma $ given by (\ref{ieq1}) is said to be in the class $\mathscr{B}_{\Sigma }^{\mu }(\alpha, \lambda ,\delta)$ if the following conditions hold for all $z,w\in \mathbb{U}$:
\begin{equation} \label{ieq21}
\left|\arg\left((1-\lambda )\left(\frac{f(z)}{z}\right)^{\mu }+\lambda f^{\prime }(z)\left(\frac{f(z)}{z}\right)^{\mu -1}+\xi\delta zf^{\prime \prime }(z)\right)\right|<\frac{\alpha \pi }{2}
\end{equation}
and
\begin{equation} \label{ieq22}
\left|\arg\left( (1-\lambda )\left(\frac{g(w)}{w}\right)^{\mu }+\lambda g^{\prime}(w)\left(\frac{g(w)}{w}\right)^{\mu -1}+\xi\delta
wg^{\prime \prime }(w)\right)\right|<\frac{\alpha \pi }{2},
\end{equation}
where the function $g(w)=f^{-1}(w)$ is defined by (\ref{ieq2}) and $\xi=\frac{2\lambda +\mu }{2\lambda +1}$.
\end{de}
\newpage
\begin{re}
Note that for $\lambda=1, \mu=1$ and $\delta =0$, the class of functions $\mathscr{B}_{\Sigma }^{1}(\alpha,1,0):=\mathscr{B}_{\Sigma }(\alpha)$ have been introduced and studied by Srivastava et al. \cite{C27}, for $\mu=1$ and $\delta =0$, the class of functions $\mathscr{B}_{\Sigma }^{1}(\alpha,\lambda,0):=\mathscr{B}_{\Sigma }(\alpha,\lambda)$ have been introduced and studied by Frasin and Aouf \cite{C28}, for $\delta =0$, the class of functions $\mathscr{B}_{\Sigma }^{\mu}(\alpha,\lambda,0):=\mathscr{B}_{\Sigma }^{\mu}(\alpha,\lambda)$ have been introduced and studied by \c{C}a\u{g}lar et al. \cite{C29}, and for $\lambda=1, \mu=0$ and $\delta =0$, we obtain the well-known class $\mathscr{B}_{\Sigma }^{0}(\alpha,1,0):=\mathcal{S}^*_\Sigma[\alpha]$ of strongly bi-starlike functions of order $\alpha$.
\end{re}

We first state and prove the following result.

\begin{thm}
\label{thm21} Let the function $f(z)$  given by (\ref{ieq1}) be in the class $\mathscr{B}_{\Sigma }^{\mu }(\alpha,\lambda,\delta)$. Then
\begin{equation} \label{theq21}
|a_{2}|\leq \frac{2\alpha }{\sqrt{\left( \lambda +\mu +2\xi \delta \right)^{2}+\alpha \left[2\lambda +\mu -\left(\lambda +2\xi \delta \right)^{2}+\left( 12-4\mu\right) \xi \delta \right]}}
\end{equation}
and
\begin{equation} \label{theq222}
 \hspace{-1.85in}|a_{3}|\leq \frac{4\alpha ^{2}}{(\lambda+\mu +2\xi\delta )^{2}}+\frac{2\alpha }{2\lambda +\mu+6\xi\delta}.
\end{equation}

\end{thm}

\begin{proof}
Let $f\in \mathscr{B}_{\Sigma }^{\mu }(\alpha,\lambda,\delta)$. From (\ref{ieq21}) and (\ref{ieq22}), we have
\begin{equation} \label{eq21}
(1-\lambda )\left(\frac{f(z)}{z}\right)^{\mu }+\lambda f^{\prime }(z)\left(\frac{f(z)}{z}\right)^{\mu -1}+\xi\delta zf^{\prime \prime }(z)=[p(z)]^{\alpha }
\end{equation}
and
\begin{equation} \label{eq22}
(1-\lambda )\left(\frac{g(w)}{w}\right)^{\mu }+\lambda g^{\prime}(w)\left(\frac{g(w)}{w}\right)^{\mu -1}+\xi\delta
wg^{\prime \prime }(w)=[q(w)]^{\alpha },
\end{equation}
where $p(z)= 1+p_{1}z+p_{2}z^{2}+\cdots$ and $q(w)=1+q_{1}w+q_{2}w^{2}+\cdots$ in $\mathcal{P}$.

Now, equating the coefficients in (\ref{eq21}) and (\ref{eq22}), we get
\begin{equation} \label{eq23}
\hspace{0.55in} \left(\lambda+\mu + 2\xi\delta \right)a_{2}=\alpha p_{1},
\end{equation}
\begin{equation} \label{eq24}
(2\lambda+\mu)\left[\left(\frac{\mu -1}{2}\right)a_{2}^{2}+\left(1+\frac{6\delta }{2\lambda +1}\right)a_{3}\right]=\alpha p_{2}+\frac{\alpha (\alpha -1)}{2}p_{1}^{2},
\end{equation}
and
\begin{equation} \label{eq25}
\hspace{0.5in} -\left(\lambda+\mu +2\xi\delta \right)a_{2}=\alpha q_{1},
\end{equation}
\begin{equation} \label{eq26}
\hspace{-.5in}(2\lambda+\mu)\left[\left(\frac{\mu +3}{2}+\frac{12\delta }{2\lambda +1}\right)a_{2}^{2}-\left(1+\frac{6\delta }{2\lambda +1}\right)a_{3}\right]=\alpha q_{2}+\frac{\alpha(\alpha -1)}{2}q_{1}^{2}.
\end{equation}

From (\ref{eq23}) and (\ref{eq25}), we obtain
\begin{equation} \label{eq27}
\hspace{0.15in} p_{1}=-q_{1},
\end{equation}
and
\begin{equation} \label{eq28}
\hspace{-0.5in} 2\left(\lambda +\mu +2\xi\delta\right)^{2} a_{2}^2=\alpha ^{2}(p_{1}^{2}+q_{1}^{2}).
\end{equation}

By adding (\ref{eq24}) to (\ref{eq26}), we get
\begin{equation} \label{eq29}
(2\lambda +\mu)\left[1+\mu +\frac{12\delta }{2\lambda +1}\right] a_{2}^{2}=\alpha (p_{2}+q_{2})+\frac{\alpha (\alpha -1)}{2}(p_{1}^{2}+q_{1}^{2}).
\end{equation}

By using (\ref{eq28}) in (\ref{eq29}), we obtain
\begin{equation} \label{eq210}
\left[ (2\lambda+\mu)(\mu +1)+12\xi\delta-\frac{(\alpha -1)}{\alpha }(\lambda+\mu +2\xi\delta)^{2}\right] a_{2}^{2}=\alpha\left(p_{2}+q_{2}\right).
\end{equation}

By considering Lemma \ref{lem} we get from (\ref{eq210}) the desired inequality (\ref{theq21}).

Next, by subtracting (\ref{eq26}) from (\ref{eq24}), we have
\begin{equation} \label{eq211}
2(2\lambda +\mu)\left(1+\frac{6\delta }{2\lambda +1}\right)a_{3}-2(2\lambda +\mu)\left(1+\frac{6\delta }{2\lambda +1}\right)a_{2}^{2}=\alpha\left(p_{2}-q_{2}\right) +\frac{\alpha(\alpha-1)}{2}\left(p_{1}^{2}-q_{1}^{2}\right).
\end{equation}

Further, in view of (\ref{eq27}), it follows from (\ref{eq211}) that
\begin{equation} \label{eq212}
a_{3}=a_{2}^{2}+\frac{\alpha}{2(2\lambda +\mu+6\xi\delta)}\left(p_{2}-q_{2}\right).
\end{equation}

By considering (\ref{eq28}) and Lemma \ref{lem}, we get from (\ref{eq212}) the desired inequality (\ref{theq222}).

This complete the proof of Theorem \ref{thm21}.
\end{proof}

Now, we would like to draw attention to some remarkable results which are obtained for some values of $\lambda, \mu$ and $\delta$ in Theorem \ref{thm21}.

If we choose $\lambda =1$, $\mu =1$ and $\delta =0$ in Theorem \ref{thm21}, we get the following consequence.

\begin{co} \cite{C27}
Let the function $f(z)$ given by (\ref{ieq1}) be in the class $\mathscr{B}_{\Sigma }(\alpha)$. Then
\begin{equation*}
\hspace{-.08in} |a_{2}|\leq \alpha \sqrt{\frac{2}{\alpha +2}}
\end{equation*}
and
\begin{equation*}
|a_{3}|\leq \frac{\alpha (3\alpha +2)}{3}.
\end{equation*}
\end{co}

If we choose $\mu =1$ and $\delta =0$ in Theorem \ref{thm21}, we get the following consequence.

\begin{co} \cite{C28}
Let the function $f(z)$  given by (\ref{ieq1}) be in the class $\mathscr{B}_{\Sigma }(\alpha,\lambda)$. Then
\begin{equation*}
|a_{2}|\leq \frac{2\alpha }{\sqrt{(\lambda +1)^{2}+\alpha(1+2\lambda -\lambda ^{2})}}
\end{equation*}
and
\begin{equation*}
\hspace{-.55in} |a_{3}|\leq \frac{4\alpha ^{2}}{(\lambda+1)^{2}}+\frac{2\alpha }{(2\lambda+1)}.
\end{equation*}
\end{co}

If we choose $\delta =0$ in Theorem \ref{thm21}, we get the following consequence.

\begin{co} \cite{C29}
Let the function $f(z)$  given by (\ref{ieq1}) be in the class $\mathscr{B}_{\Sigma }^{\mu }(\alpha,\lambda)$. Then
\begin{equation*}
|a_{2}|\leq \frac{2\alpha }{\sqrt{(\lambda+\mu)^{2}+\alpha(2\lambda+\mu -\lambda ^{2})}}
\end{equation*}
and
\begin{equation*}
\hspace{-.55in} |a_{3}|\leq \frac{4\alpha ^{2}}{(\lambda+\mu)^{2}}+\frac{2\alpha }{(2\lambda+\mu)}.
\end{equation*}
\end{co}

If we choose $\lambda =1, \mu=0$ and $\delta =0$ in Theorem \ref{thm21}, we get the following consequence.

\begin{co} \cite{C29}
Let the function $f(z)$  given by (\ref{ieq1}) be in the class $\mathcal{S}^*_\Sigma[\alpha]$. Then
\begin{equation*}
\hspace{-.2in} |a_{2}|\leq \frac{2\alpha }{\sqrt{1+\alpha}}
\end{equation*}
and
\begin{equation*}
 |a_{3}|\leq \alpha(4\alpha+1).
\end{equation*}
\end{co}

\section{Coefficient bounds for the function class $\mathscr{B}_{\Sigma }^{\mu }(\beta,\lambda,\delta)$}
This section is concerned with the coefficient bounds for the Taylor-Maclaurin coefficients $|a_2|$ and $|a_3|$ of the function $f \in \mathscr{B}_{\Sigma }^{\mu }(\beta,\lambda,\delta)$. Various known spacial cases of the main result are pointed out.

\begin{de} \label{def22}
For $\lambda \geq 1,\mu \geq 0, \delta \geq 0$ and $0 \leq \beta < 1$, a function $f\in \Sigma $ given by (\ref{ieq1}) is said to be in the class $\mathscr{B}_{\Sigma }^{\mu }(\beta, \lambda ,\delta)$ if the following conditions hold for all $z,w\in \mathbb{U}$:
\begin{equation} \label{ieq23}
\mbox{Re}\left((1-\lambda )\left(\frac{f(z)}{z}\right)^{\mu}+\lambda f^{\prime }(z)\left(\frac{f(z)}{z}\right)^{\mu -1}+\xi\delta zf^{\prime \prime }(z)\right)>\beta
\end{equation}
and
\begin{equation} \label{ieq24}
\mbox{Re} \left((1-\lambda)\left(\frac{g(w)}{w}\right)^{\mu }+\lambda g^{\prime }(w)\left(\frac{g(w)}{w}\right)^{\mu -1}+\xi\delta wg^{\prime \prime }(w)\right)>\beta,
\end{equation}
where the function $g(w)=f^{-1}(w)$ is defined by (\ref{ieq2}) and $\xi=\frac{2\lambda +\mu }{2\lambda +1}$.
\end{de}


\begin{re}
Note that for $\lambda=1, \mu=1$ and $\delta =0$, the class of functions $\mathscr{B}_{\Sigma }^{1}(\beta,1,0):=\mathscr{B}_{\Sigma }(\beta)$ have been introduced and studied by Srivastava et al. \cite{C27}, for $\mu=1$ and $\delta =0$, the class of functions $\mathscr{B}_{\Sigma }^{1}(\beta,\lambda,0):=\mathscr{B}_{\Sigma }(\beta,\lambda)$ have been introduced and studied by Frasin and Aouf \cite{C28}, for $\delta =0$, the class of functions $\mathscr{B}_{\Sigma }^{\mu}(\beta,\lambda,0):=\mathscr{B}_{\Sigma }^{\mu}(\beta,\lambda)$ have been introduced and studied by \c{C}a\u{g}lar et al. \cite{C29}, and for $\lambda=1, \mu=0$ and $\delta =0$, we obtain the well-known class $\mathscr{B}_{\Sigma }^{0}(\beta,1,0):=\mathcal{S}^*_\Sigma(\beta)$ of bi-starlike functions of order $\beta$.
\end{re}

\begin{thm}
\label{thm221} Let the function $f(z)$  given by (\ref{ieq1}) be in the class $\mathscr{B}_{\Sigma }^{\mu }(\beta,\lambda,\delta)$. Then
\begin{equation} \label{theq221}
\hspace{-1.9in} |a_{2}|\leq \min \left\{ \sqrt{\frac{4(1-\beta )}{(2\lambda +\mu )(1+\mu +\frac{12\delta }{2\lambda +1})}},\frac{2(1-\beta )}{%
\lambda +\mu +2\xi \delta}\right\}
\end{equation}
and
\begin{equation} \label{theq2222}
|a_{3}|\leq \left\{
\begin{array}{c}
\min \left\{ \frac{(1-\beta )\left(4+\frac{24\delta}{2\lambda+1}\right)}{(2\lambda +\mu+6\xi\delta)(1+\mu +\frac{12\delta }{2\lambda +1})},\frac{4(1-\beta )^{2}}{(\lambda +\mu +2\xi \delta )^{2}}+%
\frac{2(1-\beta)}{2\lambda +\mu+6\xi\delta}\right\}
;\ \  0\leq \mu <1 \\
\frac{2(1-\beta )}{2\lambda +\mu+6\xi\delta};\hspace{2.65in} \ \quad \mu
\geq 1%
\end{array}%
\right.
.
\end{equation}

\end{thm}

\begin{proof}
Let $f\in \mathscr{B}_{\Sigma }^{\mu }(\beta,\lambda,\delta)$. It follows from (\ref{ieq23}) and (\ref{ieq24}) that there exist $p,q\in \mathcal{P}$ such that
\begin{equation} \label{eq221}
(1-\lambda )\left(\frac{f(z)}{z}\right)^{\mu }+\lambda f^{\prime }(z)\left(\frac{f(z)}{z}\right)^{\mu -1}+\xi\delta zf^{\prime \prime }(z)=\beta +(1-\beta )p(z)
\end{equation}
and
\begin{equation} \label{eq222}
(1-\lambda )\left(\frac{g(w)}{w}\right)^{\mu }+\lambda g^{\prime}(w)\left(\frac{g(w)}{w}\right)^{\mu -1}+\xi\delta
wg^{\prime \prime }(w)=\beta +(1-\beta )q(w),
\end{equation}
where $p(z)= 1+p_{1}z+p_{2}z^{2}+\cdots$ and $q(w)=1+q_{1}w+q_{2}w^{2}+\cdots$ in $\mathcal{P}$.

Now, equating the coefficients in (\ref{eq221}) and (\ref{eq222}), we get
\begin{equation} \label{eq223}
\hspace{0.8in}\left(\lambda+\mu + 2\xi\delta \right)a_{2}=(1-\beta )p_{1},
\end{equation}
\begin{equation} \label{eq224}
\hspace{-0.6in} (2\lambda+\mu)\left[\left(\frac{\mu -1}{2}\right)a_{2}^{2}+\left(1+\frac{6\delta }{2\lambda +1}\right)a_{3}\right]=(1-\beta )p_{2},
\end{equation}
and
\begin{equation} \label{eq225}
\hspace{0.65in} -\left(\lambda+\mu +2\xi\delta \right)a_{2}=(1-\beta )q_{1},
\end{equation}
\begin{equation} \label{eq226}
\hspace{-1.23in} (2\lambda+\mu)\left[\left(\frac{\mu +3}{2}+\frac{12\delta }{2\lambda +1}\right)a_{2}^{2}-\left(1+\frac{6\delta }{2\lambda +1}\right)a_{3}\right]=(1-\beta )q_{2}.
\end{equation}

From (\ref{eq223}) and (\ref{eq225}), we obtain
\begin{equation} \label{eq227}
p_{1}=-q_{1},
\end{equation}
and
\begin{equation} \label{eq228}
\hspace{-.2in} 2\left(\lambda +\mu +2\xi\delta\right)^{2} a_{2}^2=(1-\beta )^{2}(p_{1}^{2}+q_{1}^{2}).
\end{equation}

By adding (\ref{eq224}) to (\ref{eq226}), we get
\begin{equation} \label{eq229}
(2\lambda +\mu)\left(1+\mu +\frac{12\delta }{2\lambda +1}\right) a_{2}^{2}=(1-\beta )(p_{2}+q_{2}).
\end{equation}

From equations (\ref{eq228}) and (\ref{eq229}), we get
\begin{equation} \label{eq5555}
\hspace{-.45in} |a_{2}|^{2}\leq\frac{(1-\beta )^{2}}{2(\lambda+\mu+2\xi\delta)^{2}}\left(|p_{1}|^{2}+|q_{1}|^{2}\right)
\end{equation}
and
\begin{equation} \label{eq55555}
|a_{2}|^{2}\leq \frac{(1-\beta )}{(2\lambda+\mu)(1+\mu+\frac{12\delta}{2\lambda +1})}\left(|p_{2}|+|q_{2}|\right),
\end{equation}
respectively.

By considering Lemma \ref{lem} we get from (\ref{eq5555}) and (\ref{eq55555}) the desired inequality (\ref{theq221}).

Next, by subtracting (\ref{eq226}) from (\ref{eq224}), we have
\begin{equation} \label{eq2211}
2(2\lambda +\mu)\left(1+\frac{6\delta }{2\lambda +1}\right)a_{3}-2(2\lambda +\mu)\left(1+\frac{6\delta }{2\lambda +1}\right)a_{2}^{2}=(1-\beta )(p_{2}-q_{2}).
\end{equation}

Further, in view of (\ref{eq228}), it follows from (\ref{eq2211}) that
\begin{equation} \label{eq2212}
a_{3}=\frac{(1-\beta )^{2}}{2\left(\lambda +\mu +2\xi\delta\right)^{2}}(p_{1}^{2}+q_{1}^{2})+\frac{(1-\beta)}{2(2\lambda +\mu+6\xi\delta)}\left(p_{2}-q_{2}\right).
\end{equation}

Applying Lemma \ref{lem} for (\ref{eq2212}), we get
\begin{equation} \label{eq6666}
|a_{3}|\leq\frac{4(1-\beta )^{2}}{\left(\lambda +\mu +2\xi\delta\right)^{2}}+\frac{2(1-\beta)}{2\lambda +\mu+6\xi\delta}.
\end{equation}

On other hand, by using the equation (\ref{eq229}) in (\ref{eq2211}), we obtain
\begin{equation} \label{eq22212}
a_{3}=\frac{1-\beta}{2(2\lambda +\mu+6\xi\delta)}\left[\left(\frac{3 +\mu+\frac{24\delta}{2\lambda+1}}{1+\mu+\frac{12\delta}{2\lambda+1}}\right)p_{2}+ \left(\frac{1-\mu}{1+\mu+\frac{12\delta}{2\lambda+1}}\right)q_{2}\right].
\end{equation}

Applying Lemma \ref{lem} for (\ref{eq22212}), we get
\begin{equation} \label{eq66666}
|a_{3}|\leq\frac{1-\beta}{2\lambda +\mu+6\xi\delta}\left(\frac{3 +\mu+\frac{24\delta}{2\lambda+1}+|1-\mu|}{1+\mu+\frac{12\delta}{2\lambda+1}}\right).
\end{equation}

Hence, by considering equation (\ref{eq6666}) we obtain from equation (\ref{eq66666}) for $0\leq \mu <1$ the first part of the desired inequality (\ref{theq2222}), and for $\mu \geq 1$ the second part of the desired inequality (\ref{theq2222}).

This completes the proof of Theorem \ref{thm221}.
\end{proof}

If we take $\lambda =1$, $\mu =1$ and $\delta =0$ in Theorem \ref{thm221}, we get the following consequence.

\begin{co} \cite{C29}
Let the function $f(z)$ given by (\ref{ieq1}) be in the class $\mathscr{B}_{\Sigma }(\beta)$. Then
\begin{equation*}
|a_{2}|\leq \left\{
\begin{array}{c}
\sqrt{\frac{2(1-\beta )}{3}};\ \ \  0\leq \beta <\frac{1}{3} \\
1-\beta; \qquad \frac{1}{3}\leq \beta <1
\end{array}%
\right.
\end{equation*}
and
\begin{equation*}
\hspace{-1in}|a_{3}|\leq \frac{2(1-\beta)}{3}.
\end{equation*}
\end{co}

If we choose $\mu =1$ and $\delta =0$ in Theorem \ref{thm221}, we get the following consequence.

\begin{co} \cite{C29}
Let the function $f(z)$  given by (\ref{ieq1}) be in the class $\mathscr{B}_{\Sigma }(\beta,\lambda)$. Then
\begin{equation*}
 |a_{2}|\leq \min \left\{ \sqrt{\frac{2(1-\beta )}{2\lambda +1}},\frac{2(1-\beta )}{%
\lambda +1}\right\}
\end{equation*}
and
\begin{equation*}
\hspace{-1.3in}|a_{3}|\leq \frac{2(1-\beta )}{2\lambda +1}.
\end{equation*}
\end{co}

If we choose $\delta =0$ in Theorem \ref{thm221}, we get the following consequence.

\begin{co} \cite{C29}
Let the function $f(z)$  given by (\ref{ieq1}) be in the class $\mathscr{B}_{\Sigma }^{\mu }(\beta,\lambda)$. Then
\begin{equation*}
\hspace{-1.1in} |a_{2}|\leq \min \left\{ \sqrt{\frac{4(1-\beta )}{(2\lambda +\mu )(\mu+1)}},\frac{2(1-\beta )}{%
(\lambda +\mu)}\right\}
\end{equation*}
and
\begin{equation*}
\hspace{-0.3in} |a_{3}|\leq \left\{
\begin{array}{c}
\min \left\{ \frac{4(1-\beta )}{(2\lambda +\mu)(\mu+1)},\frac{4(1-\beta )^{2}}{(\lambda +\mu)^{2}}+%
\frac{2(1-\beta)}{2\lambda +\mu}\right\}
;\ \  0\leq \mu <1 \\
\frac{2(1-\beta )}{2\lambda +\mu};\hspace{1.85in} \ \quad \mu
\geq 1%
\end{array}%
\right.
.
\end{equation*}
\end{co}

If we choose $\lambda =1, \mu=0$ and $\delta =0$ in Theorem \ref{thm221}, we get the following consequence.

\begin{co} \cite{C29}
Let the function $f(z)$  given by (\ref{ieq1}) be in the class $\mathcal{S}^*_\Sigma(\beta)$. Then
\begin{equation*}
\hspace{-1.55in} |a_{2}|\leq \sqrt{2(1-\beta)}
\end{equation*}
and
\begin{equation*}
\hspace{-0.3in} |a_{3}|\leq \left\{
\begin{array}{c}
2(1-\beta );\qquad \ \ \ \ 0\leq \beta <\frac{3}{4} \\
(1-\beta)(5-4\beta); \  \frac{3}{4}\leq \beta <1
\end{array}%
\right.
.
\end{equation*}
\end{co}




\begin{thebibliography}{10}

\bibitem{C1}
A.A.~Amourah, F.~Yousef, T.~Al-Hawary, M.~Darus, A certain fractional derivative operator for p-valent functions and new class of analytic functions with negative coefficients, Far East Journal of Mathematical Sciences 99(1) (2016) 75-87.

\bibitem{C4}
A.A.~Amourah, F.~Yousef, T.~Al-Hawary, M.~Darus, On a class of p−valent non-Bazilevi\u{c} functions of order $\mu+ i\beta$, International Journal of Mathematical Analysis 10(15) (2016) 701-710.

\bibitem{C5}
A.A.~Amourah, F.~Yousef, T.~Al-Hawary, M.~Darus, On H$_3$(p) Hankel determinant for certain subclass of p-valent functions, Ital. J. Pure Appl. Math 37 (2017) 611-618.

\bibitem{C26}
D.A.~Brannan, J.G.~Clunie, W.E.~Kirwan, Coefficient estimates for a class of star-like functions, Canad. J. Math. 22 (1970) 476-485.

\bibitem{C22}
D.A.~Brannan, J.G.~Clunie, Aspects of contemporary complex analysis (Proceedings of the NATO Advanced Study Institute held at the University of Durham, Durham; July 1–20, 1979), Academic Press, New York and London, 1980.

\bibitem{C25}
D.A.~Brannan, D.L.~Tan, On some classes of bi-univalent functions, Studia Univ. Babecs-Bolyai Math. 31(2) (1986) 70--77.

\bibitem{C29}
M.~\c{C}a\u{g}lar, H.~Orhan, N.~Ya\u{g}mur, Coefficient bounds for new subclasses of bi-univalent functions, Filomat 27(7) (2013) 1165-1171.

\bibitem{Duren}
P.L.~Duren, Univalent Functions, Grundlehren der Mathematischen Wissenschaften, Band 259, Springer-Verlag, New York, Berlin, Heidelberg and
Tokyo, 1983.

\bibitem{C28}
B.A.~Frasin, M.K.~Aouf, New subclasses of bi-univalent functions, Appl. Math. Lett. 24(9) (2011) 1569--1573.

\bibitem{C21}
M.~Lewin, On a coefficient problem for bi-univalent functions, Proc. Amer. Math. Soc. 18 (1967) 63--68.

\bibitem{C212}
X.-F.~Li,  A.-P.~Wang, Two new subclasses of bi-univalent functions, Internat. Math. Forum 7 (2012) 1495--1504.

\bibitem{C23}
E.~Netanyahu, The minimal distance of the image boundary from the origin and the second coefficient of a univalent function in $|z|< 1$, Arch. Rational Mech. Anal. 32 (1969) 100--112.

\bibitem{C214}
Ch.~Pommerenke, Univalent Functions, Vandenhoeck and Rupercht, G\"{o}ttingen, 1975.

\bibitem{C211}
H.M.~Srivastava, S.~Bulut, M.~\c{C}a\u{g}lar, N.~Ya\u{g}mur, Coefficient estimates for a general subclass of analytic and bi-univalent functions, Filomat 27(5) (2013) 831--842.

\bibitem{C213}
H.M.~Srivastava, S.~Gaboury, F.~Ghanim, Coefficient estimates for a general subclass of analytic and bi-univalent functions of the Ma–Minda type, Revista de la Real Academia de Ciencias Exactas, Físicas y Naturales. Serie A. Matemáticas 27(5) (2017) 1--12.

\bibitem{C27}
H.M.~Srivastava, A.K.~Mishra, P.~Gochhayat, Certain subclasses of analytic and bi-univalent functions, Appl. Math. Lett. 23(10) (2010) 1188--1192.

\bibitem{C24}
D.L.~Tan, Coefficicent estimates for bi-univalent functions, Chin. Ann. Math. Ser. A 5 (1984) 559--568.

\bibitem{C210}
Q.-H.~Xu, Y.-C.~Gui, H.M.~Srivastava, Coefficient estimates for a certain subclass of analytic and bi-univalent functions, Appl. Math. Lett. 25 (2012) 990--994.

\bibitem{C2}
F.~Yousef, A.A.~Amourah, M.~Darus, Differential sandwich theorems for p-valent functions associated with a certain generalized differential operator and integral operator, Italian Journal of Pure and Applied Mathematics 36 (2016) 543-556.

\bibitem{C3}
F.~Yousef, B.A.~Frasin, T.~Al-Hawary, Fekete-Szeg\"{o} Inequality for Analytic and Bi-univalent Functions Subordinate to Chebyshev Polynomials, arXiv preprint arXiv:1801.09531 (2018).

\end{thebibliography}
\end{document}